\documentclass[11pt]{amsart}

\usepackage{amsmath}
\usepackage{amssymb}
\usepackage{amsthm}
\usepackage{amscd, amsfonts}
\usepackage[dvips]{epsfig}
\usepackage{verbatim}

\usepackage{epsf}
\usepackage[usenames]{color}
\usepackage[bookmarksnumbered,pdfpagelabels=true,plainpages=false,colorlinks=true,
            linkcolor=black,citecolor=black,urlcolor=blue]{hyperref}

\theoremstyle{plain}
\newtheorem{theorem}{Theorem}[section]
\newtheorem{lemma}[theorem]{Lemma}
\newtheorem{prop}[theorem]{Proposition}

\theoremstyle{definition}

\numberwithin{equation}{section}

\def\XXint#1#2#3{{\setbox0=\hbox{$#1{#2#3}{\int}$}
\vcenter{\hbox{$#2#3$}}\kern-.5\wd0}}

\title[Szego kernel of some Reinhardt domains]{On the asymptotics of the on-diagonal Szeg\"o kernel of certain Reinhardt domains}

\author[Karami]{Arash Karami}
\address{Department of Mathematics\\
Johns Hopkins University\\ Baltimore, MD 21210, USA}
\email{akarami@math.jhu.edu}

\author[Pingali]{Vamsi Pingali}
\address{Krieger 412, Department of Mathematics\\
Johns Hopkins University\\ Baltimore, MD 21210, USA}
\email{vpingali@math.jhu.edu}

\begin{document}

\begin{abstract}
We compute the leading and sub-leading terms in the asymptotic expansion of the Szeg\"o kernel on the diagonal of a class of pseudoconvex Reinhardt domains whose boundaries are endowed with a general class of smooth measures. We do so by relating it to a Bergman kernel over projective space.
\end{abstract}

\maketitle

\section{Introduction}
\indent The Szeg\"o and Bergman kernels play an important role in complex analysis \cite{MG,SB}, physics \cite{CJ}, random complex geometry \cite{AK,BS,TT}, and K\"ahler geometry \cite{SZ}. It is not possible to compute them for general domains/manifolds although special cases have been considered \cite{GH}. \\
\indent They have asymptotic expansions under very general assumptions. Computing even the first few terms of the asymptotic expansions can have many applications \cite{AK}. In fact we were motivated to compute them for the Szeg\"o kernel because they appear in the study of random complex polynomials \cite{AK}. In this paper we compute the leading and sub-leading terms for the Szeg\"o kernel of Reinhardt domains whose defining functions are homogeneous. Previously, similar computations were done using the Boutet De Monvel-Sj\"ostrand parametrix and the stationary phase method (see \cite{AK} and the references therein). Our computation relies on a relationship between the Bergman and Szeg\"o kernels. While this relationship is very closely related to the well known one for unit disc bundles \cite{SZ}, the novelty lies in our allowance of more general measures on the boundary of the domain.\\
\indent \emph{Acknowledgements}: We thank Bernie Shiffman, George Marinescu, and Steve Zelditch for useful discussions.
\section{Summary of results}
\indent In this section we review basic definitions and theorems, and present the main result of the paper. \\
\indent Let $\rho^{-1}(-\infty, 1) = \Omega \subset \mathbb{C}^{n+1}$ be a domain with a smooth boundary $M=\rho^{-1}(1)$ where $\rho$ is a smooth function on $\mathbb{C}^{n+1}$ having non-zero gradient on $M$. Assume that $M$ is endowed with a smooth volume form $\mu$. The holomorphic tangent space $T$ of $M$ is defined as the complex kernel of the $(1,0)$ form $\partial \rho$. The Hardy space $H(M,\mu)$ is defined as the space of all $L^2$ functions (with respect to $\mu$) $f$ on $M$ such that $\bar{\partial} \tilde{f} \vert_{T} = 0$ for any extension $\tilde{f}$ of $f$ to a neighbourhood of $M$.  \\
\indent The Szeg\"o kernel $S$ of $(\Omega, \mu)$ is a function $S: \Omega \times \Omega \rightarrow \mathbb{C}$ such that the orthogonal projection $\Pi : L^2(\Omega) \cap C^0 (\bar{\Omega}) \rightarrow H(M,\mu)$ is given by $\Pi (f) (z)= \displaystyle \int _{M} S(z,w)f(w) \mu$. It is not hard to see that the Szeg\"o kernel exists, is smooth away from the diagonal, and maybe computed as $S(z,w) = \displaystyle \sum _{i} \phi _i (z) \overline{\phi_i (w)}$ where $\phi_i$ form an orthonormal basis of $H(M,\mu)$ \cite{SK}. If $\Omega$ is a strictly pseudoconvex domain (i.e. the defining $\rho$ may be chosen to be proper and strictly plursubharmonic) then it maybe proven that the $\phi_i$ maybe chosen as the restriction of holomorphic functions on $\Omega$. It is not possible to compute the Szeg\"o kernel in general. However, in special cases one may get more information. If $\rho$ and $\mu$ are $S^1$ invariant (i.e. complete circular domains) then $H(M,\mu)$ consists of $S^1$ invariant functions. It may be proven the decomposition $H(M,\mu) = \displaystyle \oplus _{k=0}^{\infty} H_k (M,\mu)$ holds \cite{AK}(where the $H_k$ are irreducible representations of $S^1$). Thus $S=\displaystyle \sum _{k=0} ^{\infty} \Pi_k$. The $\Pi_k$ are called partial Szeg\"o kernels. There exists an asymptotic expansion for it in terms of $k$ for certain strictly pseudoconvex domains.\\
\indent In this paper we restrict ourselves to domains $\Omega = \rho^{-1}[0,1)$  where $\rho = f(\vert x_0 \vert , \vert x_1 \vert \ldots)$ where $f$ is homogeneous of order $l$ (as a consequence these domains are Reinhardt). For example, $\rho=\sum_{i=0}^{2}\vert x_i \vert^4 +  (\vert x_0 x_1\vert^2 + \vert x_0 x_2\vert^2 + \vert x_1 x_2\vert^2)$ on $\mathbb{C}^{3}$. It may be proven that monomials of order $k$ form an orthogonal basis for $H_{k}(M,\mu)$ with respect to the inner product in equation\cite{AK,SK}. Let $J=(j_{0},\dots,j_{n})$, $x^{J}$ be a monomial of degree n, $|J|=n$, and $x=(x_{0},\dots,x_{n})$. The partial Szeg\"o kernel is
\begin{equation}
\Pi_{k}(x,x)=\sum_{|J|=n}\frac{|x^{J}|^{2}}{\langle x^{J},x^{J} \rangle _{\mu}}\:\: J=(j_{0},\dots,j_{n}).
\label{03}
\end{equation}
where
\begin{equation}
(\langle x^{J},x^{J} \rangle_{\mu}=\int_{M}|x_{0}|^{2j_{0}}\dots|x_{n}|^{2j_{n}}\mu.
\label{04}
\end{equation}
\indent For the sake of brevity we define
\begin{equation}
\psi:\mathbb{C}^{n+1}-{(0,0,\ldots)}\longrightarrow\mathbb{R}\:\:\: \psi(x_{0},\dots,x_{n})=\rho^{-1/l}(x_{0},\dots,x_{n}).
\end{equation}
\indent Now we define the Bergman kernel. Let $(Y,\omega)$ be an  $n$-dimensional compact K\"ahler manifold and $(L^k\otimes E, h_L ^k \otimes h_E)$ be a hermitian holomorphic vector bundle over $Y$ (where $L$ is a line bundle). The Bergman kernel $B_k$ is defined as a section $B_k$ of End(E) such that the orthogonal projection
$$\Pi_{k} : L^2 (Y,L^k\otimes E,\omega, h_L ^k \otimes h_E) \rightarrow H(Y,L^k\otimes E,\omega, h_L ^k \otimes h_E),$$
is $\Pi_{k}(f)(z) = \displaystyle \int _Y B_k (z,w)f(w) \frac{\omega^n}{n!}$ where $H(.)$ is the space of holomorphic sections of $L^k \otimes E$. If $(L,h_L)$ is ample then there is an asymptotic expansion for the Bergman kernel \cite{Ma, RBeBBeSj}:
\begin{theorem}
There exist smooth coefficients $b_{r}(y)\in End(E)_{y}$ which are polynomials in $\omega$ ,$-\partial\bar{\partial}\ln(h_{E})$ and their derivatives and reciprocals of linear combinations of eigenvalues of $\dot{R_{L}}_{i}^{k}=-(\frac{\partial^{2}\log h_{L}}{\partial z_{i}\partial\bar{z}_{j}})(w^{k\bar{j}})$ at $y$, such that for any $j,l\in\mathbb{N}$, there exist $C_{k,l}$ such that for any $k\in\mathbb{N}$,
\begin{equation}
|B_{k}(y,y)-\sum_{r=0}^{l}b_{r}(y)k^{n-r}|_{L^{j}(Y)}\leq C_{l,j}k^{n-l-1},
\end{equation}
and
\begin{equation}
b_{0}(y)=\det\left(\frac{\dot{R_{L}}}{2\pi}\right)Id_{E},
\end{equation}
and
\begin{equation}
b_{1}=b_{0}\left(\frac{r_{\omega}}{2}-\Delta_{\omega}(\log(b_{0}))+\sqrt{-1}\Lambda_{\omega}(-\partial\bar{\partial}\ln(h_{E})\right).
\end{equation}
where $r_{\omega}$ is the scalar curvature of $(Y,\omega)$ in the complex geometry convention (i.e. $\omega^{i\bar{j}} \mathrm{Ric}_{i\bar{j}}$), $\Delta _{\omega} f = -\omega ^{i\bar{j}} f_{i\bar{j}}$ and $\sqrt{-1}\Lambda _{\omega} \partial \bar{\partial} \ln(h_E) = \omega ^{i\bar{j}} (\ln(h_E)) _{i\bar{j}}$.
\label{bergasym}
\end{theorem}
For ease of notation we introduce the following matrices,
\begin{equation}
H(\rho)=(\frac{\partial^{2}\rho}{\partial x_{i}\partial\overline{x}_{j}})_{0\leq i,j\leq n}\:\:\text{and}\:\: H_{0}(\rho)=(\frac{\partial^{2}\rho}{\partial x_{i}\partial\overline{x}_{j}})_{1\leq i,j\leq n},
\label{c4}
\end{equation}
\begin{equation}
\nabla\rho(x)=(\frac{\partial\rho}{\partial\overline{x}_{0}},\dots,\frac{\partial\rho}{\partial\overline{x}_{n}})\:\:\text{and}
\:\:\nabla_{0}\rho(x)=(\frac{\partial\rho}{\partial\overline{x}_{1}},\dots,\frac{\partial\rho}{\partial\overline{x}_{n}}).
\label{c5}
\end{equation}

\indent Finally we state our main theorem
\begin{theorem}
If $x=(x_{0},\dots,x_{n})\in\Omega$ where $\Omega=\rho^{-1}[0,1)$ is defined earlier, then the first two terms of the asymptotic expansion of the reproducing kernel of the projection map
\begin{equation}
\Pi_{k}:L^{2}(\Omega,\mu)\longrightarrow H_{k}(M),
\end{equation}
are
\begin{equation}
a_{0}(x,x)=
\left(\frac{2}{l}\right)^{n+2}\frac{\det(H(\rho))}{2\pi^{n+1}e^{u(x\psi(x))}\psi^{2n-l(n+1)}\sqrt{\left(\frac{\partial \psi}{\partial \vert x_0 \vert}\right)^2 + \left(\frac{\partial \psi}{\partial \vert x_1 \vert}\right)^2+\ldots}}, \nonumber
\end{equation}
and
\begin{gather}
a_1 (x,x)= \frac{a_0}{4} \left( 2n(n+1) + 2\displaystyle  \sum _{\mu=0}^n \vert x \vert^2 \frac{\partial ^2 \ln(a_0)}{\partial x_{\mu} \partial \overline{x_{\mu}}} \right) ,\nonumber
\end{gather}
where $\Pi_{k}(x,x)=a_{0}(x,x)k^{n+1}+a_{1}(x,x)k^{n}+O(k^{n-1}) + \ldots$.
\label{main}
\end{theorem}
\section{Proof of the main theorem}
\indent In order to compute the asymptotic expansion of $S$, we relate the partial Szeg\"{o} kernel $\Pi_k$ of $\Omega$ with the Bergman kernel $B_k$ of $(\mathcal{O}(k)\otimes E, h_{k} \otimes h_{E})$ over $(\mathbb{CP}^n,\omega _{FS})$ where $E$ is a trivial line bundle. The Fubini-study metric in a chart is $\omega_{FS} = i\partial \bar{\partial} \ln (1+\vert z \vert^2)$. After doing so we use the asymptotic expansion of the Bergman kernel as given in \cite{Ma},\cite{RBeBBeSj}.\\
\indent In order to simplify computations define
\begin{equation}
\psi:\mathbb{C}^{n+1}-{(0,0,\ldots)}\longrightarrow\mathbb{R}\:\:\: \psi(x_{0},\dots,x_{n})=\rho^{-1/l}(x_{0},\dots,x_{n}).
\end{equation}
\indent The gradient of $\rho$ is nonzero on $M=\rho^{-1}(1)$ and hence $M$ is a $2n+1$-dimensional manifold admitting a natural $S^{1}$ action : $x \rightarrow e^{i\theta} x$. Identifying $\tilde{M}=M/S^{1}$, we have a quotient map
$$\tilde{p}:M\longrightarrow\tilde{M}=M/S^{1} \:\:\tilde{\pi}(x)=\tilde{x}.$$
\begin{lemma}
$\tilde{M}$ is an $n$-dimensional complex manifold biholomorphic to $\mathbb{CP}^n$ via a map $\pi:\mathbb{CP}^n\rightarrow M$.
\end{lemma}
\begin{proof}
Write $\tilde{M}=\cup_{i=0}^{n}\tilde{M}_{i}$ where $\tilde{M}_{i}=\{(x_{0},\dots,x_{n})\in \tilde{M} : x_{i}\neq 0\}$ and define
$$\phi_{0}:\mathbb{C}^{n}\longrightarrow\tilde{M}_{0}\:\: \phi_{0}(w_{1},\dots,w_{n})=\psi(1,w_{1},\dots,w_{n})(1,w_{1},\dots,w_{n}),$$
and similarly we define $\phi_{i}$. The $\phi_{i}$ are certainly homeomorphisms. To show that $\phi_{i}^{-1}o\phi_{j}$ is holomorphic, we do so for $i=1, j=0$ and the same proof works for different $i,j$.
\begin{equation}
\begin{split}
\phi_{1}^{-1}o\phi_{0}(w_{1},\dots,w_{n})&=\phi_{1}^{-1}(\psi(1,w_{1},\dots,w_{n})(1,w_{1},\dots,w_{n}))\\
&=\phi_{1}^{-1}(\psi(1,1/w_{1},\dots,w_{n}/w_{1})(1/w_{1},1,\dots,w_{n}/w_{1}))=(1/w_{1},\dots,w_{n}/w_{1}).\\
\end{split}
\end{equation}
Hence $\phi_{1}^{-1}o\phi_{0}$ is holomorphic on $\phi_{0}^{-1}(\tilde{M}_{0}\cap\tilde{M}_{1})$. \\
Define the map
$$\pi:\mathbb{CP}^n\longrightarrow M\:\:\text{where}$$
$$\pi([x_{0},\dots,x_{n}])=\psi(x_{0},\dots,x_{n})(x_{0},\dots,x_{n}).$$
It is easy to see that $\tilde{p}\circ\pi$ is a well-defined biholomorphism onto $\tilde{M}$. \end{proof}
\begin{lemma}
The defining function $\rho$ induces a Hermitian metric $h$ on the hyperplane section bundle $(O(1),\mathbb{CP}^{n})$ such that
\begin{eqnarray}
h_{\beta}([z_0,z_1,\ldots,1,z_{\beta+1},\ldots])&=&\psi^2(z_0,z_1,\ldots,1,z_{\beta+1},\ldots)\:\:\text{on}\:  U_{\beta}\:\text{and}\: , \nonumber \\
\Vert e_{\alpha} \Vert_{h^k} ^2 &=& \pi^{*} \vert x^{\alpha} \vert^2.
\end{eqnarray}
 \label{05}
 \end{lemma}
\begin{proof}
 The collection of $h_{\beta}$ does patch up to give a Hermitian metric on $O(1)$. Indeed, the transition functions $g_{\beta \gamma}$ for $O(1)$ on $U_{\beta}\cap U_{\gamma}$ are equal to $\frac{x_{\beta}}{x_{\gamma}}$, $g_{\beta \gamma}([x_{0},\dots,x_{n}])=\frac{x_{\beta}}{x_{\gamma}}$ and by definition of $h_{\beta}$ we have
$$\frac{h_{\beta}}{h_{\gamma}}=\frac{|x_{\beta}|^{2}}{|x_{\gamma}|^{2}}=|g_{\beta \gamma}|^{2}.$$ Moreover on $U_{\beta}$,
\begin{eqnarray}
\Vert e_{\alpha} \Vert_{h^k} ^2 &=& \vert z_0 ^{\alpha _0} z_1 ^{\alpha _1} \ldots z_{\beta-1}^{\alpha_{\beta-1}} z_{\beta+1} ^{\alpha _{\beta+1}} \ldots \vert ^2 h_{\beta} ^k \nonumber \\
&=& \vert z_0 ^{\alpha _0} z_1 ^{\alpha _1} \ldots z_{\beta-1}^{\alpha_{\beta-1}} z_{\beta+1} ^{\alpha _{\beta+1}} \ldots \vert ^2 \psi ^{2k}(z_0,z_1,\ldots,1,z_{\beta+1},\ldots) \nonumber \\
&=& \pi ^{*} \vert x^{\alpha} \vert^2.
\end{eqnarray}
\end{proof}
Let $\mu_{ind}$ be the volume form on $M$ induced from the Euclidean metric on $\mathbb{C}^{n+1}$. An arbitrary $S^1$-invariant volume form $\mu$ on $M$ may be written as $\mu = e^u \mu_{ind}$ where $u$ is a smooth function on $\tilde{M}$. We want to find a function $h_E$ on $\mathbb{CP}^n$ such that $\displaystyle \int _{\mathbb{CP}^n} \pi^{*}(\mathcal{F}) h_E \frac{\omega_{FS}^n}{n!} = \int _M \mathcal{F} \mu$ for every $S^1$-invariant function $\mathcal{F}$ on $M$. To this end, notice that
\begin{equation}
\begin{split}
&\pi^{*}dx_{0}=d\psi(1,z_{1},\dots,z_{n})\:\:\text{and}\\
&\pi^{*}dx_{i}=d(z_{i}\psi(1,z_{1},\dots,z_{n}))=z_{i} d\psi+\psi dz_{i} \:\:\text{when i=1,\dots,n}.\\
\end{split}
\label{06}
\end{equation}
\begin{lemma}
There is a smooth function $h_{E}$ on $\mathbb{CP}^{n}$ satisfying
$$\displaystyle h_{E}([y_0,y_{1},\dots,y_{n}])=2\pi \pi^{*}(e^u)\vert y \vert ^{2n+2} (\frac{\psi^{2}(\vert y_0 \vert, \vert y_1 \vert, \ldots)}{2})^{n} \sqrt{\left (\frac{\partial \psi}{\partial \vert y_0 \vert} \right ) ^2 + \left (\frac{\partial \psi}{\partial \vert y_1 \vert} \right ) ^2  + \ldots},$$
and
$$\displaystyle \int _{\mathbb{CP}^n} \pi^{*}(\mathcal{F}) h_E \frac{\omega_{FS}^n}{n!} = \int _M \mathcal{F} \mu ,$$
\label{07}
for every $S^1$-invariant function $\mathcal{F}$ on $M$.
\end{lemma}
\begin{proof}
We denote coordinates in $\mathbb{C}^{n+1}$ by $(x_0, x_1, \ldots)$, and homogeneous coordinates on $\mathbb{CP}^n$ by $[y_0,\ldots]$. Define $z_i = \frac{y_i}{y_0}$ when $y_0 \neq 0$ and let $x_i = R_i e^{\sqrt{-1} \Theta _i}$. Let us assume that $R_0$ may be solved for as a function $R_0 = f(R_1, \ldots, R_n)$ on a domain $D$. The metric on $M$ induced from the Euclidean one on $\mathbb{C}^{n+1}$ is
\begin{eqnarray}
g &=& df \otimes df + f^2 d\Theta_0 ^2 + \sum _{i=1} ^{n} dR_i ^2 + R_i ^2 d\Theta _i ^2 \nonumber \\
&=& \frac{\partial f}{\partial R_i} \frac{\partial f}{\partial R_j} dR_i dR_j +  f^2 d\Theta_0 ^2 + \sum _{i=1} ^{n} dR_i ^2 + R_i ^2 d\Theta _i ^2.
\label{indmet}
\end{eqnarray}
It is easy to see that the corresponding volume form is
$$\mu_{ind} = f\sqrt{1+\displaystyle \sum _{i=1}^{n}\left(\frac{\partial f}{\partial R_i}\right)^2} d\Theta _0 R_1 dR_1 d\Theta _1 R_2 dR_2 d\Theta _2 \ldots . $$
 We may write this expression more invariantly as follows :
\begin{eqnarray}
\psi (f, R_1, R_2,\ldots) &=& 1 \nonumber \\
\frac{\partial \psi}{\partial R_0}\frac{\partial f}{\partial R_i} + \frac{\partial \psi}{\partial R_i} &=& 0 \nonumber \\
\mu_{ind} &=& R_0 \frac{\sqrt{\displaystyle \sum _{i=0} ^n \left(\frac{\partial \psi}{\partial R_i}\right)^2}}{\vert \frac{\partial \psi}{\partial R_0}\vert} d\Theta _0 R_1 dR_1 d\Theta _1 R_2 dR_2 d\Theta _2 \ldots
\end{eqnarray}
Thus the integral over $M$ of a continuous $S^1$-invariant function $\mathcal{F}$ is easily seen to be
\begin{gather}
 \displaystyle 2\pi \int _D \mathcal{F} R_0 \frac{\sqrt{\displaystyle \sum _{i=0} ^n \left(\frac{\partial \psi}{\partial R_i}\right)^2}}{\vert \frac{\partial \psi}{\partial R_0}\vert} R_1 dR_1 d\Theta _1 R_2 dR_2 d\Theta _2 \ldots
 \end{gather}
    Our task is reduced to calculating $\pi^{*} (R_1 dR_1 d\Theta _1 R_2 dR_2 d\Theta _2 \ldots) $ on $U_0$ which is the same as $\pi ^{*} \frac{\tilde{\omega} ^n}{n!}$ where $\tilde{\omega} = \displaystyle \frac{\sqrt{-1}}{2} \sum_{i=1}^{n} dx_i \wedge d\bar{x}_i$ on $\mathbb{C}^{n+1}$.  For $z=[1,z_{1},\dots,z_{n}]\in U_{0}$ we have
\begin{equation}
\begin{split}
\pi^{*}\tilde{\omega}&=\frac{\sqrt{-1}}{2}\sum_{i=1}^{n}\pi^{*}dx_{i}\wedge\pi^{*}d\overline{x_{i}}\\
&=\frac{\sqrt{-1}}{2}\sum_{i=1}^{n}(z_{i} d\psi+\psi dz_{i})\wedge(\overline{z_{i} d\psi+\psi dz_{i}})\\
&=\frac{\sqrt{-1}}{2}\sum_{i=1}^{n}z_{i}\overline{\psi}d\psi\wedge d\overline{z_{i}}+\frac{\sqrt{-1}}{2}\sum_{i=1}^{n}\overline{z_{i}}\psi d z_{i}\wedge d\overline{\psi}+\psi^{2}\frac{\sqrt{-1}}{2}\sum_{i=1}^{n}d z_{i}\wedge d\overline{z_{i}}\\
&=\frac{\sqrt{-1}}{2}\psi\sum_{i=1}^{n}d\psi\wedge(z_{i}d\overline{z}_{i}-\overline{z}_{i}d z_{i})+\frac{\sqrt{-1}}{2}\psi^{2}\sum_{i=1}^{n}d z_{i}\wedge d\overline{z}_{i}.
\end{split}
\end{equation}
If we use cylindrical polar coordinates $z_{j}=r_{j}e^{\sqrt{-1}\theta_{j}}$, then
\begin{equation}
z_{i}d\overline{z}_{i}-\overline{z}_{i}d z_{i}=-2\sqrt{-1}r_{i}^2d\theta_{i}.
\end{equation}
We then compute $d\psi$ :
\begin{equation}
d\psi=\displaystyle \sum_{j=1}^{n}\frac{\partial\psi}{\partial r_{j}} d r_{j}.
\end{equation}
So in cylindrical polar coordinates we have
\begin{equation}
\begin{split}
\displaystyle \pi^{*}\tilde{\omega}&=\psi\sum_{i=1}^{n}\sum_{j=1}^{n}(\frac{\partial\psi}{\partial r_{j}}r_{i})d r_{j}\wedge d\theta_{i}+\psi^{2}\sum_{i=1}^{n}r_{i}d r_{i}\wedge d\theta_{i}\\
&=\sum_{i=1}^{n}\sum_{j=1}^{n}(r_i\frac{\partial\psi}{\partial r_{j}}\psi+\delta_{ij}\psi^{2})r_{i}d r_{j}\wedge d\theta_{i},\\
\end{split}
\end{equation}
thus implying that
\begin{equation}
\begin{split}
\pi^{*}\frac{\tilde{\omega}^{n}}{n!}&=\left(\psi^{2n}+\psi^{2n-1}\sum_{i=1}^{n}r_i\frac{\partial\psi}{\partial r_{i}}\right)r_{1}dr_1 d\theta_1 r_2 dr_2 d\theta_2 \ldots\\
&=\displaystyle \left (-2^{n}\psi^{2n-1}\frac{\partial \psi (\vert y_0 \vert,\vert y_1 \vert, \ldots)}{\partial \vert y_0 \vert}|_{(1,r_1,\ldots)}\right)\left(1+|z|^{2}\right)^{n+1}\frac{\omega_{FS}^{n}}{n!},\\
\end{split}
\label{int}
\end{equation}
where the last equality follows from the homogenity of $\psi$. We want to choose $h_E$ such that
\begin{gather}
\displaystyle \int _{\mathbb{CP}^n} \pi^{*} \mathcal{F} h_E \frac{\omega _{FS} ^n}{n!} = \displaystyle \int _{\mathbb{C}^n} \pi^{*} (2\pi \int _D \mathcal{F} R_0 \frac{\sqrt{\displaystyle \sum _{i=0} ^n \left(\frac{\partial \psi}{\partial R_i}\right)^2}}{\vert \frac{\partial \psi}{\partial R_0}\vert} R_1 dR_1 d\Theta _1 R_2 dR_2 d\Theta _2 \ldots) \nonumber .\\
\end{gather}
Writing equation \ref{int} in terms of $y_0, y_1, \ldots$ and substituting in the above expression we see that
\begin{equation}
\displaystyle h_{E}([y_0,y_{1},\dots,y_{n}])=2\pi \pi^{*}(e^u)\vert y \vert ^{2n+2} (\frac{\psi^{2}(\vert y_0 \vert, \vert y_1 \vert, \ldots)}{2})^{n} \sqrt{\left (\frac{\partial \psi}{\partial \vert y_0 \vert} \right ) ^2 + \left (\frac{\partial \psi}{\partial \vert y_1 \vert} \right ) ^2  + \ldots}
\end{equation}
\end{proof}
    The function $h_{E}$ is smooth on $\mathbb{CP}^{n}$ and so it may be seen as a Hermitian metric on the trivial bundle $(E,\mathbb{CP}^{n})$. Finally we have the following useful relationship,
\begin{lemma}
The Bergman kernel $B_k$ of $(\mathbb{CP}^n,\omega_{FS},\mathcal{O}(k)\otimes E, h_k \otimes h_E)$ restricted to $\mathbb{CP}^n \times \mathbb{CP}^n$ is related to the partial Szego kernel $\Pi_k$ of $(M,\mu)$ restricted to $M\times M$ as
\begin{eqnarray}
 \pi^{*}\Pi_k(x,x) &=& \frac{B_k ([x],[x])}{h_E}.
\end{eqnarray}
\label{bergszego}
\end{lemma}
\emph{Proof}: The Bergman kernel is
\begin{gather}
B_k ([x],[x]) = \displaystyle \sum _{\vert \alpha \vert =k} \frac{\Vert e_{\alpha} \Vert_{h_k \otimes h_E} ^2}{\int _{\mathbb{CP}^n}\Vert e_{\alpha} \Vert_{h_k \otimes h_E} ^2 \frac{\omega_{FS} ^n}{n!}} \nonumber \\
=\sum _{\vert \alpha \vert =k} \frac{\pi^{*}(\vert x^{\alpha} \vert^2) h_E}{\langle x^{\alpha}, x^{\alpha} \rangle _{\mu}} \nonumber \\
= \pi^{*} \Pi_k (x,x) h_E .\nonumber
\end{gather}
\subsection{Curvature of the Hermitian metrics}
\indent Here we compute the curvature of the Hermitian metric $h$. For the sake of brevity we denote
\begin{equation}
\rho_{1}=\frac{\partial\rho}{\partial z_{1}},\dots,\rho_{n}=\frac{\partial\rho}{\partial z_{n}},\: \rho_{i\bar{j}}=\frac{\partial^{2}\rho}{\partial z_{i}\partial\bar{z}_{j}},
\end{equation}
\begin{equation}
A=(\rho\rho_{i\bar{j}})_{1\leq i,j\leq n}\:\:\text{, and a vector}\:\: v=(\overline{\rho}_{1},\dots,\overline{\rho}_{n}).
\end{equation}
Therefore the curvature of $h$ at the point $[1,z_{1},\dots,z_{n}]\in U_{0}$ is
\begin{equation}
\begin{split}
\Theta_{h}&=-\partial\bar{\partial}\ln h = -2\partial\bar{\partial}(\psi(1,z_{1},\dots,z_{n}))\\
&=\frac{2}{l}\partial\bar{\partial}\ln \rho(1,z_{1},\dots,z_{n})=\frac{2}{l}\sum(\frac{\rho\rho_{i\bar{j}}-\rho_{i}\rho_{\bar{j}}}{\rho^{2}})dz_{i}\wedge d\bar{z}_{j}.\\
\end{split}
\label{curv}
\end{equation}
The preceding expression for $\Theta_h$ may be used to prove its positivity. Indeed,
\begin{lemma}
The curvature of the Hermitian metric $h$ on the hyperplane section bundle $(O(1),\mathbb{CP}^{n})$ is positive.
\label{c6}
\end{lemma}
\begin{proof}
\indent We use the plurisubharmonicity of the defining function $\rho$ in the following form:
\begin{eqnarray}
\partial_x \bar{\partial _x} \rho ((a,\vec{w}), (\bar{a},\overline{\vec{w}})) >0 ,\label{psh}
\end{eqnarray}
for every $a \in \mathbb{C}$ and $\vec{w} \in \mathbb{C}^n$. Notice that
\begin{eqnarray}
\partial_x \bar{\partial _x} \rho &=& \frac{\partial ^2}{\partial x_0 \partial \bar{x}_0}(\vert x_0 \vert^l \rho(1,z_1,\ldots))dx_0 \wedge d\bar{x}_0 + \frac{\partial ^2}{\partial x_0 \partial \bar{z}_i}(\vert x_0 \vert^l \rho(1,z_1,\ldots))dx_0 \wedge d\bar{z}_i \nonumber \\ &+& \frac{\partial ^2}{\partial z_i \partial \bar{x}_0}(\vert x_0 \vert^l \rho(1,z_1,\ldots))dz_i \wedge d\bar{x}_0 + \vert x_0 \vert^l \partial_z \bar{\partial}_z \rho.  \nonumber \\
\partial_x \bar{\partial _x} \rho ((a,\vec{w}), (\bar{a},\overline{\vec{w}})) &=& \vert x_0 \vert ^{l-2} \displaystyle \left | al + \frac{x_o}{\rho} \sum_i \frac{\partial \rho}{\partial z_i} w_i \right | ^2 + \vert x_0 \vert ^l \partial _z \bar{\partial}_z \rho (\vec{w},\overline{\vec{w}})-\frac{\vert x_0 \vert^{l}}{\rho} \left |\sum_i \frac{\partial \rho}{\partial z_i} w_i \right | ^2 .\nonumber
\end{eqnarray}
Choosing $a=\frac{-1}{l} \sum \frac{X_0}{\rho}\frac{\partial \rho}{\partial z_i} w_i$, and using expressions \ref{curv} and \ref{psh}, we see that the curvature $\Theta_h$ is positive.
\end{proof}
The following computation is useful :
\begin{lemma}
The determinant of the curvature matrix on the coordinate chart $U_0$ is
$$\det(\Theta_{h})=(\frac{2}{l\rho^{2}})^{n}\det(A)(1-vA^{-1}v^{*}).$$
\end{lemma}
\begin{proof}
Recall that
\begin{equation}
\frac{l}{2}(\rho^{2}\Theta_{h})_{j\bar{j}}=(\rho\rho_{i\bar{j}}-\rho_{i}\rho_{\bar{j}}).
\end{equation}
 If we let $A_{i}$ be the $i$-th column of the matrix $A$ then
\begin{equation}
\frac{l}{2}\rho^{2}\Theta_{h}=(A_{1}-\rho_{1}v^*,\dots,A_{n}-\rho_{n}v^*),
\end{equation}
Hence
\begin{equation}
\begin{split}
\frac{l^{n}\rho^{2n}}{2^{n}}\det(\Theta_{h})&=\det(A)-\rho_{1}\det(v^{*},A_{2},\dots,A_{n})-\dots-\rho_{n}\det(A_{1},\dots,v^{*})\\
&=\det(A)(1-\sum_{i=1}^{n}\rho_{i}\frac{\det(\tilde{A_{i}})}{\det(A)})\:\text{Cramer's rule},\\
&=\det(A)(1-vA^{-1}v^{*}).\\
\end{split}
\end{equation}
Therefore
\begin{equation}
\det(\Theta_{h})=(\frac{2}{l\rho^{2}})^{n}\det(A)(1-vA^{-1}v^{*}).
\label{1}
\end{equation}
\end{proof}
It is easy to see that
\begin{prop}
If $f(\vert x_0 \vert, \vert x_1 \vert,\ldots)$ is a homogeneous function of order $l$,
\begin{eqnarray}
\frac{\vert x_0 \vert^l}{x_0}f _{i} &=& \frac{\partial f}{\partial x_i}, \nonumber \\
\frac{\vert x_0 \vert^l}{\bar{x_0}}f _{\bar{i}} &=& \frac{\partial f}{\partial \bar{x}_i}, \nonumber \\
\vert x_0 \vert^{l-2} f_{i\bar{j}} &=& \frac{\partial^2 f}{\partial x_i \partial \bar{x}_j}. \nonumber
\end{eqnarray}
\label{xcood}
\end{prop}
We then rewrite equation \eqref{1}  in the x-coordinates :
\begin{lemma}
The determinant of the curvature matrix on $U_0$ of the Hermitian metric $h$ in the x-coordinates is
\begin{equation}
\det(\Theta_{h})=\left(\frac{2}{l}\right)^{n+2} \left(\frac{\vert x_0 \vert^2}{\rho}\right)^{n+1} \det(H(\rho)).
\end{equation}
\label{3}
\end{lemma}
\begin{proof}
Using proposition \ref{xcood} we may write $A$ and $v$ in the x-coordinates. That is,
\begin{equation}
A=\rho(\rho_{i\bar{j}})=\frac{\rho(x)}{|x_{0}|^{2l-2}}H_{0}(\rho),
\end{equation}
and
\begin{equation}
v=\frac{1}{|x_{0}|^{l-1}}\nabla_{0}\rho.
\end{equation}
Then
$$vA^{-1}v^{*}=\frac{1}{\rho(x)}(\nabla_{0}\rho) H_{0}^{-1}(\rho)(\nabla_{0}\rho)^{*}.$$
So we have
\begin{equation}
\begin{split}
\det(\Theta_{h})&=\left(\frac{2}{l\rho^{2}}\right)^{n}\det(A)(1-vA^{-1}v^{*})\\
&=\left(2\frac{|x_{0}|^{2l}}{l\rho^{2}(x)}\right)^{n}\left(\frac{\rho(x)}{|x_{0}|^{2l-2}}\right)^{n}\det(H_{0}(\rho))(1-\frac{1}{\rho(x)}(\nabla_{0}\rho) H_{0}^{-1}(\rho)(\nabla_{0}\rho)^{*})\\
&=\left(\frac{2|x_{0}|^{2}}{l\rho(x)}\right)^{n}\det(H_{0}(\rho))(1-\frac{1}{\rho(x)}(\nabla_{0}\rho) H_{0}^{-1}(\rho)(\nabla_{0}\rho)^{*})\\
&=\left(\frac{2|x_{0}|^{2}}{l\rho}\right)^{n}\frac{1}{\rho}\det(H_{0}(\rho))(\rho-(\nabla_{0}\rho) H_{0}^{-1}(\rho)(\nabla_{0}\rho)^{*}).\\
\end{split}
\label{interm}
\end{equation}
We then proceed to write $\det(H(\rho))$ in terms of $H_0 (\rho)$. To do this we use the homogenity of $\rho$ in the form
\begin{eqnarray}
\displaystyle \sum _{i=1} ^n x_i \frac{\partial \rho}{\partial x_i} + x_0\frac{\partial \rho}{\partial x_0} &=& \frac{l\rho}{2}, \nonumber \\
\displaystyle \sum _{i=1} ^n \bar{x_i} \frac{\partial \rho}{\partial \bar{x_i}} + \bar{x_0}\frac{\partial \rho}{\partial \bar{x_0}} &=& \frac{l\rho}{2}.
\label{hom}
\end{eqnarray}
Differentiating the first equation in \ref{hom} with respect to $\bar{x_j}$ we see that
\begin{gather}
\displaystyle \sum _{i=1} ^{n} x_i \frac{\partial^2 \rho}{\partial x_i \partial \bar{x_j}} + x_0 \frac{\partial^{2}\rho}{\partial x_0 \partial \bar{x_j}} = \frac{l}{2} \frac{\partial \rho}{\partial \bar{x_j}} ,\nonumber \\
\frac{\partial ^2 \rho}{\partial x_0 \partial \bar{x_0}} = \frac{l^2 \rho}{4 \vert x_0 \vert^2} - \frac{l}{2\vert x_0 \vert^2} \sum _{i=1} ^{n} \bar{x_i} \frac{\partial \rho}{\partial \bar{x_i}} - \frac{x_i}{x_0} \frac{\partial^2 \rho}{\partial x_i \partial \bar{x_0}}.
\label{sec}
\end{gather}
So by using the equation \ref{sec} we have:
\begin{equation}
\frac{\partial ^{2}\rho}{\partial x_0 \partial \bar{x_j}}=\frac{l}{2x_0} \frac{\partial \rho}{\partial \bar{x_j}}- \sum _{i=1} ^n \frac{x_i}{x_0} \frac{\partial ^2 \rho}{\partial x_i \partial \bar{x_j}}.
\label{hessr}
\end{equation}
The equations \ref{sec}, \ref{hessr} imply that the determinant of the matrix $H(\rho)$ is (using row operations)
\begin{gather}
\left | \begin{array}{cc}
(\frac{l^2 \rho}{4 \vert x_0 \vert^2} - \frac{l}{2\vert x_0 \vert^2} \sum _{i=1} ^{n} \bar{x_i} \frac{\partial \rho}{\partial \bar{x_i}} - \sum _{i=1} ^{n} \frac{x_i}{x_0} \frac{\partial^2 \rho}{\partial x_i \partial \bar{x_0}})  & \dots  (\frac{l}{2x_0} \frac{\partial \rho}{\partial \bar{x_n}} - \sum _{i=1} ^n \frac{x_i}{x_0} \frac{\partial ^2 \rho}{\partial x_i \partial \bar{x_n}}) \\
\displaystyle\vdots \\ (\frac{\partial ^2 \rho} {\partial x_n \partial \bar{x_0}}) & \displaystyle H_0\\
\end{array} \right | \nonumber \\ \\
= \left |
\begin{array}{cc}
 (\frac{l^2\rho}{4 \vert x_0 \vert^2} - \frac{l}{2\vert x_0 \vert^2} \sum _{i=1} ^{n} \bar{x_i} \frac{\partial \rho}{\partial \bar{x_i}})  &\dots  \frac{l}{2x_0} \frac{\partial \rho}{\partial \bar{x_n}} \\ \vdots\\
\frac{l}{2\bar{x_0}} \frac{\partial \rho}{\partial x_n} -   \sum _{i=1} ^n \frac{\bar{x_i}}{\bar{x_0}} \frac{\partial ^2 \rho}{\partial \bar{x_i} \partial x_n} & \displaystyle H_0
\end{array} \right |
= \frac{l}{2x_0}\left |
\begin{array}{cc}
 (\frac{l\rho}{2 \bar{x_0}} - \frac{1}{\bar{x_0}} \sum _{i=1} ^{n} \bar{x_i} \frac{\partial \rho}{\partial \bar{x_i}})  &\dots  \frac{\partial \rho}{\partial \bar{x_n}} \\ \vdots\\
\frac{l}{2\bar{x_0}} \frac{\partial \rho}{\partial x_n} -   \sum _{i=1} ^n \frac{\bar{x_i}}{\bar{x_0}} \frac{\partial ^2 \rho}{\partial \bar{x_i} \partial x_n} & \displaystyle H_0
\end{array} \right |. \nonumber \\
=  \frac{l}{4 \vert x_0 \vert^2}\left |
\begin{array}{cc}
 \rho  &\dots  \frac{\partial \rho}{\partial \bar{x_n}} \\ \vdots\\
 \frac{\partial \rho}{\partial x_n}  & \displaystyle H_0
\end{array} \right |. \nonumber
\end{gather}
At this point we split the determinant as
\begin{gather}
\frac{l}{4\vert x_0 \vert^2} \det(A + B)
= \frac{l}{4\vert x_0 \vert^2} \det(A) \det(I + A^{-1}B), \nonumber
\end{gather}
where the matrices $A$ and $B$ are
\begin{gather}
A =  \left [ \begin{array}{cc}
\rho & 0 \\
0 & H_0
\end{array} \right ],\nonumber \\
B = \left [ \begin{array}{cc}
0 &\dots \frac{\partial \rho}{\partial \bar{x_n}}  \\ \vdots\\
\frac{\partial \rho}{\partial x_n} & \displaystyle 0_{n\times n}
\end{array}
\right ].\nonumber
\end{gather}
Hence the determinant equals $\frac{l^2}{4\vert x_0 \vert^2} \det(H_0) \det(I+A^{-1}B)$ which is easily evaluated to be
$$\frac{l^2}{4\vert x_0 \vert^2} \det(H_0) (\rho-(\nabla_{0}\rho) H_{0}^{-1}(\rho)(\nabla_{0}\rho)^{*}).$$
 This in conjunction with equation \ref{interm} proves the lemma.
\end{proof}

\subsection{Terms of the asymptotic expansion}
Finally we may prove theorem \ref{main} which is stated once again for the reader's convenience :
\begin{theorem}
If $x=(x_{0},\dots,x_{n})\in\Omega$ where $\Omega=\rho^{-1}[0,1)$ is defined earlier, then the first two terms of the asymptotic expansion of the reproducing kernel of the projection map
\begin{equation}
\Pi_{k}:L^{2}(\Omega,\mu)\longrightarrow H_{k}(M),
\end{equation}
are
\begin{equation}
a_{0}(x,x)=
\left(\frac{2}{l}\right)^{n+2}\frac{\det(H(\rho))}{2\pi^{n+1}e^{u(x\psi(x))}\psi^{2n-l(n+1)}\sqrt{\left(\frac{\partial \psi}{\partial \vert x_0 \vert}\right)^2 + \left(\frac{\partial \psi}{\partial \vert x_1 \vert}\right)^2+\ldots}}, \nonumber
\end{equation}
and
\begin{gather}
a_1 (x,x)= \frac{a_0}{4} \left( 2n(n+1) + 2\displaystyle  \sum _{\mu=0}^n \vert x \vert^2 \frac{\partial ^2 \ln(a_0)}{\partial x_{\mu} \partial \overline{x_{\mu}}} \right) ,\nonumber
\end{gather}
where $\Pi_{k}(x,x)=a_{0}(x,x)k^{n+1}+a_{1}(x,x)k^{n}+O(k^{n-1}) + \ldots$.
\end{theorem}
\begin{proof}
By using equation lemma \ref{bergszego}, theorem \ref{bergasym}, lemma \ref{07}, and lemma \ref{3} we have
\begin{eqnarray}
a_{0}(x,x)&=&\frac{1}{h_E}\frac{1}{(2\pi)^{n}}\frac{\det(\Theta_{h})}{\det(\Theta_{FS})}\\
& &=\frac{1}{h_E}\frac{1}{(2\pi)^{n}}\frac{\left(\frac{2}{l}\right)^{n+2} \left(\frac{\vert x_0 \vert^2}{\rho}\right)^{n+1} \det(H(\rho)) }{\frac{|x_{0}|^{2n+2}}{(|x_{0}|^{2}+\dots+|x_{n}|^{2})^{n+1}}}\\
& &=\frac{1}{h_{E}}\frac{1}{(2\pi
)^{n}}\left(\frac{2}{l}\right)^{n+2}\left(\frac{|x|^{2}}{\rho}\right)^{n+1}\det(H(\rho))\\
& &=(\frac{2}{l})^{n+2}\frac{\det(H(\rho))}{2\pi^{n+1}e^{u(x\psi(x))}\psi^{2n-l(n+1)}\vert d\psi\vert}
\end{eqnarray}
An easy application of lemma \ref{bergszego}, theorem \ref{bergasym}, lemma \ref{07}, and proposition \ref{xcood} shows that
\begin{eqnarray}
a_1 (x,x)&=& \frac{a_0}{4} (2n(n+1) + 4\displaystyle \left( \sum _{i=1}^n \vert x \vert^2 \frac{\partial ^2 \ln(a_{0}h)}{\partial x_i \partial \bar{x_i}} + \frac{\vert x \vert^2}{\vert x_0 \vert^2} \bar{x_i} x_j \frac{\partial ^2 \ln(a_{0}h)}{\partial x_i \partial \bar{x_j}} \right)  \nonumber \\ & & - 4\left( \sum _{i=1}^n \vert x \vert^2 \frac{\partial ^2 \ln(h)}{\partial x_i \partial \bar{x_i}} + \frac{\vert x \vert^2}{\vert x_0 \vert^2} \bar{x_i} x_j \frac{\partial ^2 \ln(h)}{\partial x_i \partial \bar{x_j}} \right) ). \nonumber
\end{eqnarray}
Using equations \ref{sec} with $l=0$ yields the desired formula for $a_1$. Hence we have
\begin{gather}
a_1 (x,x)= \frac{a_0}{4} \left( 2n(n+1) + 2\displaystyle  \sum _{\mu=0}^n \vert x \vert^2 \frac{\partial ^2 \ln(a_0)}{\partial x_{\mu} \partial \overline{x_{\mu}}} \right) .\nonumber
\end{gather}
\end{proof}

\end{document}